\documentclass[10pt,reqno]{amsart}
\usepackage[comma,numbers,sort]{natbib}
\usepackage{amsmath}
\usepackage{amsfonts}
\usepackage{amssymb}
\usepackage{amsthm}
\usepackage{dsfont}
\usepackage{bbm}
\usepackage{hyperref}
\usepackage{paralist}
\usepackage{enumitem}
\usepackage{subref}
\usepackage{xcolor}
\usepackage{mathtools}
\usepackage{thmtools}
\usepackage{thm-restate}
\usepackage[margin=1in]{geometry}
\mathtoolsset{showonlyrefs}
\usepackage{optprog}

\usepackage[utf8]{inputenc} 
\usepackage[T1]{fontenc}    
\usepackage{nicefrac}       

\DeclareMathOperator{\tr}{\mathsf{tr}}
\DeclareMathOperator{\adj}{\mathsf{adj}}
\DeclareMathOperator{\rk}{\mathsf{rank}}
\DeclareMathOperator{\sign}{\mathsf{sign}}
\let\det\relax
\DeclareMathOperator{\det}{\mathsf{det}}

\newcommand\set[2]{\left\{ {#1} \ : \ {#2} \right\}}

\newcommand\R{\mathbb{R}}

\newcommand\Rsym[1]{\mathcal{S}^{#1}}

\newcommand\x{\mathbf{x}}
\newcommand\y{\mathbf{y}}

\newcommand\C{\mathbf{C}}
\renewcommand{\u}{\mathbf{u}}
\newcommand{\A}{\mathbf{A}}
\newcommand{\M}{\mathbf{M}}
\newcommand{\X}{\mathbf{X}}
\newcommand{\Y}{\mathbf{Y}}
\newcommand{\V}{\mathbf{V}}
\newcommand{\m}{\mathbf{m}}
\newcommand{{\ve}}{\mathbf{v}}

\newcommand{\B}{\mathbf{B}}
\newcommand{\Q}{\mathbf{Q}}

\newcommand\bz{\mathbf{0}}

\renewcommand\int{\mathrm{int}\,}

\renewcommand\S[1]{\mathcal{S}^{#1}_+}

\newcommand\Snk[2]{\mathcal{S}^{#1,#2}_+}

\renewcommand\u{\mathbf{u}}

\newcommand\innerproduct[2]{\left<#1, #2\right>}

\newcommand\define{\,:=\,}
\newcommand\sm{\setminus}

\newcommand{\calS}{\mathcal{S}}
\newcommand\Sminors[2]{\mathcal{S}^{#1,#2}_+}
\newtheorem{thm}{Theorem}
\newtheorem{cor}[thm]{Corollary}
\newtheorem{lem}[thm]{Lemma}
\newtheorem{prop}[thm]{Proposition}

\theoremstyle{definition}
\newtheorem{dfn}{Definition}

\newcommand\ignore[1]{}

\newcommand{\w}{\mathbf{w}}

\title{On generators of $k$-PSD closures of the positive semidefinite cone}
\author{Avinash Bhardwaj}
\address{Avinash Bhardwaj,Indian Institute of Technology Bombay, Mumbai, India 400076}
\email{abhardwaj@iitb.ac.in}
\author{Vishnu Narayanan}
\address{Vishnu Narayanan, Indian Institute of Technology Bombay, Mumbai, India 400076}
\email{vishnu@iitb.ac.in}
\author{Abhishek Pathapati}
\address{Abhishek Pathapati, Indian Institute of Technology Bombay, Mumbai, India 400076}
\email{apathapati@iitb.ac.in}

\begin{document}
\begin{abstract}
    Positive semidefinite (PSD) cone is the cone of positive semidefinite matrices, and is the object of interest in semidefinite programming (SDP). A computational efficient approximation of the PSD cone is the $k$-PSD closure, $1 \leq k < n$, cone of $n\times n$ real symmetric matrices such that all of their $k\times k$ principal submatrices are positive semidefinite. For $k=1$, one obtains a polyhedral approximation, while $k=2$ yields a second order conic (SOC) approximation of the PSD cone. These approximations of the PSD cone have been used extensively in real-world applications such as AC Optimal Power Flow (ACOPF) to address computational inefficiencies where SDP relaxations are utilized for convexification the non-convexities. However a theoretical discussion about the geometry of these conic approximations of the PSD cone is rather sparse. In this short communication, we attempt to provide a characterization of some family of generators of the aforementioned conic approximations.
\end{abstract}

\maketitle
\section{Introduction}
Consider the vector space of $n\times n$ symmetric matrices, $\R^{n(n+1)/2}$. In this vector space, a semidefinite programming (SDP) problem involves minimizing a linear function over the intersection of the convex cone of positive semidefinite matrices with an affine subspace. In particular,
\begin{optprog*}
    minimize & \objective{\innerproduct{C}{X}_F}\\
    (SDP) \qquad subject to & \innerproduct{A_i}{X}_F = b_i & \, i = 1,2,\ldots n\label{prob:SDP}\\
    & X \in \S{n} \hspace{6mm} .&
\end{optprog*}
where $\S{n} \define \set{X \in \R^{n\times n}}{X = X^\top,\, \y^\top X\y \geq 0~~\forall~~\y \in \R}$\ignore{$\S{n} := \set{X \in \Rsym{n}}{X \succeq 0}$} denotes the cone of all $n\times n$ real symmetric positive semidefinite matrices. SDPs have been studied extensively in optimization literature, and have been used as convex relaxations of several non-convex optimization problems \cite{goemans1994879, oustry2022certified,sojoudi2014exactness, dey2022cutting}. As large-scale SDPs are difficult to solve computationally, a general approach is to consider relaxations of the problem where instead of the variable $X$, only $k \times k$ principal submatrices of $X$, for $2 \leq k < n$ are constrained to be positive semidefinite~\cite{kimkojima2001,kocuk2016strong}.
\begin{dfn}($k$-PSD closure \cite{blekherman2022sparse})
    Given positive integers $n$ and $k$ where $2 \leq k \leq n$, the $k$-PSD closure $\Sminors{n}{k}$ is the set of all $n \times n$ symmetric real matrices where all $k \times k$ principal submatrices are PSD (positive semidefinite).
\end{dfn}

In certain cases, these relaxations yield near accurate solutions to the (SDP) problem \cite{kocuk2016strong, sojoudi2014exactness, dey2022cutting}. The question pertaining to quantifying the closeness of approximation of $\calS^{n,k}_+$ and $\calS^{n}_+$ has been explored recently \cite{blekherman2022sparse,blekherman2022hyperbolic}. These works primarily utilize a projection metric to quantify the closeness of approximation. \ignore{To further improve the projection metric bounds of \citet{blekherman2022sparse}, in a subsequent work Blekherman et al. \cite{blekherman2022hyperbolic} spectral analysis of $\calS^{n,k}_+$.}  The convex cone of eigenvalues of matrices in $\calS^{n,k}_+$ is called the \emph{hyperbolicity cone} and its properties are studied in \cite{blekherman2022hyperbolic,kozhasov2023eigenvalues}. Using certain properties of the hyperbolicity cone, \citet{blekherman2022hyperbolic} improve upon the Frobenius distance bounds of an earlier work \cite{blekherman2022sparse}. 

A different direction of research studies the number of $k\times k$ psd cones needed to approximate the $\calS^n_+$ positive semidefinite cone \cite{fawzi2019representing, fawzi2021polyhedral, song2023approximations}. It has been established that the number of cones needed to approximate the $\calS^n_+$ cone is exponential if $k$ is \textit{small} compared to $n$. Consequently, these works conclude that the $k$-PSD closure relaxations can be \textit{arbitrarily} bad. There have been some explorations in the literature relating to the exactness of these relaxations under sparse affine constraints \cite{grone1984positive, fukuda2001exploiting, sojoudi2014exactness, vandenberghe2015chordal}.

Our motivation to study these relaxations of the positive semidefinite cone primarily stems from the following: Even though, in the worst case scenario, a $k-$ PSD closure $(k > 1)$ yields an arbitrarily bad relaxation  of the PSD cone, the two cones do share some common faces/generators. In that case, given a linear objective function, understanding the geometry of these relaxations might provide some insight into the quantifying the quality of the relaxation $\Sminors{n}{k}$ for $\S{n}$, which may not be arbitrarily bad. In fact, it may give a geometric insight into why some relaxations are exact.\ignore{Our motivation stems from the fact that understanding the geometry of these relaxations might provide some insight into the quantifying the closeness of the relaxation $\calS^{n,k}_+$ to $\calS^n_+$ with respect to some lineart objective function\ignore{problem of why these relaxations i.e, $\calS^{n,k}_+$ give near optimal solutions}.} Specifically, we study the extreme rays of the relaxed convex cones $\calS_+^{n,k}$ and try to understand the number of new generators that are added to the generators of the PSD cone, as $k$ varies. Gouveia et al. \cite{gouveia2022sums} characterize the extreme rays of $\calS_+^{n,n-1}$ and provide characterization of the dual cone of $\calS_+^{n,k}$ called the factor-width cone. Additionally, they completely characterize the extreme rays of $\calS^{4,3}_+$. {\color{red}\ignore{The dual cone of $\calS^{n,k}_+$ is called the \textit{factor-width} $FW^n_k$ cone of factor-width $k$. A structural result about \textit{NLS}, non-singular locally singular, matrices of $\calS^{n,k}_+$ is proven in \cite{blekherman2022hyperbolic}.}} In the following, we characterize the extreme rays of $\mathcal{S}_+^{n,2}$ and provide a different proof for characterization of the extreme rays of $\mathcal{S}_+^{n,n-1}$ to the one already given in \cite{gouveia2022sums}.

The manuscript is arranged as follows: Section \ref{sec:results} we state our  results. In Section \ref{sec:preliminaries} we define the notation, definitions and some preliminary lemmas. Finally in Section \ref{sec:proofs} we provide proofs of stated results
\section{Main results}\label{sec:results}
\begin{thm}\labeli[prop:extrays] For $\mathbf{X} \in \Rsym{n}$, the following hold:
    \begin{enumerate}
        \item\labelii[prop:extrays1] $\X$ spans an extreme ray of $\Snk{n}{k}$ if all $k \times k$ principal submatrices of $\X$ are rank-1 and positive semidefinite.
        \item\labelii[prop:extrays2] $\X$ spans an extreme ray of $\Snk{n}{2}$ if and only if all $2 \times 2$ principal submatrices of $\X$ are rank-1 and positive semidefinite.
        \item\labelii[prop:extrays3] $\X$ spans an extreme ray of $\Snk{n}{n-1}$ if and only if all $(n-1) \times (n-1)$ principal submatrices of $\X$ are positive semidefinite and
        \begin{enumerate}
            \item either rank-1,
            \item or rank-$(n-2)$ such that $\det(\X) < 0$.
        \end{enumerate}
    \end{enumerate}
\end{thm}
The corollary follows from Theorem \ref{prop:extrays3}.
\begin{restatable}{cor}{Gnkextrays}
\label{cor:Gnkextrays}
    For $n\geq 3$ and $2\leq k\leq n-1$, \(G(n,k) = \dfrac{k}{k-1}\mathbb{I} - \dfrac{1}{k-1}\mathbf{1}\mathbf{1}^{\top}\) spans an extreme ray of \(\Sminors{n}{k}\).    
\end{restatable}
A related structure termed as non-singular locally singular ($\emph{NLS}$) matrices was introduced by \citet{blekherman2022hyperbolic} as following: 
\begin{dfn}[NLS Matrix~\cite{blekherman2022hyperbolic}]
A matrix $\M\in \calS^{n,k}_+$ is (\emph{NLS}) for $n\geq 5$ and $3 \leq k\leq n-2$ if all the $k\times k$ principal minors are zero and $\det(\M)\neq 0$. 
\end{dfn}
NLS matrices are \textit{diagonally congruent} to $G(n,k)$ matrices which are used derive distance bounds between $\calS^{n,k}_+$ and $\calS^n_+$ \cite{blekherman2022hyperbolic}. The following result with respect to NLS matrices follows from Corollary \ref{cor:Gnkextrays}.

\begin{restatable}{cor}{NLSextremerays}
\label{cor:NLSextremerays}
    An NLS matrix $\M \in \calS^{n,k}_+$, for $n\geq 5$ and $3\leq k\leq n-2$, spans an extreme ray of $\calS^{n,k}_+$.
\end{restatable}
\section{Preliminaries}\label{sec:preliminaries}
\subsection{Notation}
$\R^n$ denotes the $n-$dimensional Euclidean space and the set $\Rsym{n} \define \set{X \in \R^{n\times n}}{X = X^\top}$ defines the set of $n\times n$ symmetric real matrices. By $[n]$, we denote the index set $\{1,2,\ldots,n\}$. Additionally, $\Q \in \Rsym{n}$ is a positive semidefinite matrix, denoted by $\Q \succeq 0$, if and only if $\Q \in \S{n}$. For $\M \in \Rsym{n}$ we denote by $\M_I$, the $k\times k$ principal sub-matrix of $\M$ indexed by columns and rows $I \subseteq [n]$, $|I| = k$. Analogously, for $\ve \in \R^n$, $\ve_J$ denotes the subvector of $\ve$ corresponding to the index set $J\subseteq [n]$. Additionally, for $i, j \in [n]$, $\M_{ij}$ represents the submatrix of $\M$ indexed by rows $[n]\sm\{i\}$ and columns $[n]\sm\{j\}$. It follows that, for some $i\in [n]$, $\M_{[n]\sm\{i\}}$, the $(n-1)\times (n-1)$ principal sub-matrix of $\M$ can also be represented as $\M_{ii}$. The element of a matrix $\M$ indexed by $i^{th}$ row and $j^{th}$ column is represented as $m_{ij}$. Similarly, $v_j$ is the $j^{th}$ element of the vector $\ve_j$. We denote by $\tr(\M)$ and $\adj(\M)$ the trace and adjugate of the matrix $\M$, respectively. $\mathbb{I}_{n}$ and $\mathbf{1}_n$ represent the $n \times n$ identity matrix and the $n$ dimensional vector of all ones, respectively. $\sign$ represents signum function. $\langle\,. \,,\,.\rangle_F$ denotes Frobenius inner product on the set of symmetric matrices.

\subsubsection{Preliminary Lemmas}
\begin{lem}\label{lem:2}[Matrix Determinant Lemma, \cite{harville1998matrix}]
Given a matrix $\M \in \Rsym{n}$ and $\u,\ve \in \R^n$, $$\det(\M + \u{\ve}^\top) = \det(\M) + \tr(\adj(\M)~\u{\ve}^\top)$$.
\end{lem}
\begin{lem}\label{lem:3}
Consider $\M\in \R^{n\times n}$. If $\exists ~i \in [n]$ such that $\rk(\M_{ii}) = r$. Then $\rk(\M) \geq r$.
\end{lem}


\begin{lem}\label{lem:4}
Consider $\M\in \mathcal{S}^n$. Let \(J = [n]\sm\{i\}\) and $\m^{(i)}$ be the $i^{th}$ row of $\M$. Then $$\det(\M) = m_{ii}\det(\M_{J}) - \m^{(i)\top}_{J} \adj(\M_{J})~\m_{J}^{(i)}.$$
\end{lem}
\begin{proof}
Without loss of generality, let $i = 1$. From Laplace's equation, by expanding along the first row we have,
\begin{align}\label{prop_eq:6_1}
    \det(\M) &= \sum_{j=1}^n (-1)^{(1+j)} m_{1j}\det(\M_{1j}) = m_{11}\det(\M_{11})+\sum_{j=2}^n (-1)^{(1+j)} m_{1j}\det(\M_{1j}).
\end{align}
Let $\m^1$ be the first column of the matrix $ \M $ and $ J \define [n]\sm\{1\} $.
One can obtain the sub-matrix $\M_{1j}$ from $\M_{11}$ by performing two operations on $\M_{11}$. 
\begin{enumerate}
    \item The first operation is to replace $(j-1)^{th}$ column of $ \M_{11}$ with $\m_{J_1}^1$. Let the matrix obtained from the first operation be $\M'_{1j}$. 
Then $\M'_{1j}$ is same as $\M_{11}$ except the $(j-1)^{th}$ column. Therefore $(j-1)^{th}$ row of $\adj(\M_{11})$ and $\adj(\M'_{1j})$ are the same.

\begin{align*}
    \adj(\M'_{1j})\M'_{1j} = \det(\M'_{1j}) \mathbb{I}_{n-1}.
\end{align*}
Here $ \mathbb{I}_{n-1} $ is the $ (n-1)\times (n-1) $ identity matrix.
Let the $(j-1)^{th}$ row of $\adj(\M_{11})$ be $(\mathbf{a}^{j-1})^{\top}$. We have,
\begin{equation}\label{prop_eq:6_2}
    \det(\M'_{1j}) = (\mathbf{a}^{j-1})^{\top}\m_{J_1}^{1}.
\end{equation}
\item The second operation is to sequentially interchange columns of the matrix $\M'_{ij}$ i.e, interchange the $(j-1)^{th}$ column with the $ (j-2)^{th} $, then the $ (j-2)^{th} $ with the $ (j-3)^{th} $, and so on till $ \m_{J_1}^1 $ is the first column. Each column interchange changes the sign of the determinant by $-1$. As there are $ j-2 $ column interchanges we have,
    \begin{equation}\label{prop_eq:6_3}
        \det(\M_{1j}) = (-1)^{j} \det(\M'_{1j}).
    \end{equation}
\end{enumerate}
From \eqref{prop_eq:6_2} and \eqref{prop_eq:6_3}, we obtain
\begin{equation}\label{prop_eq:6_4}
    \det(\M_{1j}) = (-1)^j \mathbf{a}_{j-1} \m_{J_1}^1.
\end{equation}

Combining \eqref{prop_eq:6_1} and \eqref{prop_eq:6_4},
\begin{align*}
    \det(\M) &= m_{11}\det(\M_{J}) +\sum_{j=2}^n (-1)^{1+j}\m_{1j} \det(\M_{1j})\\
     &= m_{11}\det(\M_{J}) + \sum_{j=2}^n (-1)^{1+j}\m_{1j} (-1)^j (\mathbf{a}^{j-1})^{\top}\m^{1}_{J}\\
     &= m_{11}\det(\M_{J})- (\m_{J}^{1})^{\top} \adj(\M_{11})\m^{1}_{J}.
\end{align*}
\end{proof}
\begin{lem}[\cite{horn2012matrix}]\label{lem:5}
    Consider $\M\in\mathcal{S}^n$. If $\mathsf{rank}(\M) = n-1$ then $\mathsf{rank}(\adj(M)) = 1$. Furthermore, if $\M\in \mathcal{S}_+^n$ then $\adj(\M) \succeq \bz$ i.e, $\adj(\M) = \u\u^{\top}$ for some $\u\in\R^n$.
\end{lem}
\begin{lem}\label{lem:6}
  Let $J_{1},J_{2},\ldots,J_{k}\subset[n]$\ignore{$J_{1},J_{2},\ldots,J_{k}\subset[n]$ where $|J_i|=n-1~\forall i \in [k]$. If $k\leq n-1$ and $J_i\neq J_j\,\, \forall i\neq j$,}
  be pairwise distinct subsets of $[n]$ with $|J_i| = n-1$.
  Then $\displaystyle \bigcap_{i=1}^k J_{i} \neq \emptyset$.
\end{lem}
\begin{proof}
    As $|J_i| = n-1$ \ignore{it can be written as} we can write $J_i = [n]\backslash\{j_i\}$ for some $j_i \in [n]$. Let $J = [n]\backslash\{j_1,j_2,\ldots,j_k\}$. Observe that $J\subset J_{i}\,\,\forall i\in [k]$. As $|J|\geq 1$, we have that $J\subseteq \displaystyle \bigcap_{j=1}^k J_{i} \neq \emptyset$.
\end{proof}
\section{Proofs}\label{sec:proofs}
\subsection{Proof Strategy}
The idea for characterizing the extreme rays of $\calS^{n,2}_+$ is to construct a perturbation of the matrix $\M \in \calS^{n,2}_+$ if any of the $2\times 2$ principal sub-matrix has rank $2$. And if a $\M\in \calS_+^{n,2}$ such that all its $2\times 2$ principal minors are rank 1 then the result follows. 

The idea for characterizing the extreme rays of $\mathcal{S}^{n,n-1}_+$ is to characterize the ranks of the $(n-1)\times (n-1)$ principal sub-matrices of a $\M \in \mathcal{S}^{n,n-1}_+$ which span an extreme ray. We first eliminate the cases for which the $\M \in \mathcal{S}_+^{n,n-1}$ under given ranks of $(n-1)\times (n-1)$ principal sub-matrices.
The characterization of the $(n-1)\times (n-1)$ principal sub-matrices is divided into as follows,
\begin{itemize}
    \item There exists at least one $(n-1) \times (n-1)$ principal sub-matrix whose rank is less than or equal to $n-3$. We prove that these matrices do not span an extreme rays.
    \item There exists at least one $(n-1) \times (n-1)$ principal sub-matrix whose rank is $n-1$. We prove this by constructing a perturbation of $\X$ for $\M \in \calS_+^{n,n-1}$ which satisfies the aforementioned hypothesis.
\end{itemize}

The final case is if all the $(n-1)\times (n-1)$ principal sub-matrices have rank $n-2$. If all the $(n-1)\times (n-1)$ sub-matrices have rank $n-2$ and $\mathsf{det}(\M) \geq 0$ then this matrix is positive semidefinite matrix of rank $n-2$, therefore it doesn't span an extreme ray of $\mathcal{S}^{n,n-1}_+$.
Finally, we prove that if a matrix $\M\in \calS^{n,n-1}_+$ and its $(n-1)\times (n-1)$ principal sub-matrices have rank $n-2$ and $\mathsf{det}(\M)< 0$ then it spans an extreme ray of $\mathcal{S}^{n,n-1}_+$.
\subsection{Proof of Theorem \ref{prop:extrays1}}
\begin{proof}
    Let $\X \in \Snk{n}{k}$ with all $k \times k$ principal submatrices being rank-1 positive semidefinite, and suppose that $\X = \frac{1}{2}(\Y + \mathbf{Z})$ for some $\Y, \mathbf{Z} \in \Snk{n}{k}$. Additionally, let $\mathcal{I}_k$ denote the set of all index subsets of $\{1,2,\ldots, n\}$ of cardinality $k$ and for $I \in \mathcal{I}_k$, $X_I = [x_{ij}]_{i,j \in I}$. For any $I \in \mathcal{I}_k$, we have 
     \[{\X_I} = 
     \frac{1}{2}{\Y_I}
     +
     \frac{1}{2}{\mathbf{Z}_I}.
     \]
     As $\X_I, \Y_I$, and $\mathbf{Z}_I$ all belong to $\S{k}$, and $\X_I$ is of rank 1, we see that $\Y_I$ and $\mathbf{Z}_I$ must be scalar multiples of $\X_I$. By iterating $I$ over all elements of $\mathcal{I}_k$, we see that $\Y$ and $\mathbf{Z}$ must be scalar multiples of $\X$.
\end{proof}
\subsection{Proof of Theorem \ref{prop:extrays2}}
\begin{proof}
     From definition of $\Snk{n}{2}$, all $2 \times 2$ principal submatrices of $\X$ are positive semidefinite. Now, if a particular minor $\begin{bmatrix}
                 x_{ii} & x_{ij}\\
                 x_{ij} & x_{jj}
              \end{bmatrix}$ 
     has rank 2, i.e., $x_{ii}x_{jj} - x_{ij}^2 > 0$, then $\X = \frac{1}{2}(\Y + \mathbf{Z})$, where $\Y = \X - \varepsilon(F_{ij} + F_{ji})$, and $\mathbf{Z} = \X + \varepsilon(F_{ij} + F_{ji})$ for some sufficiently small $\varepsilon > 0$. Here, $F_{ij}$ is the matrix with 1 in the $(i,j)$-th entry and 0 elsewhere. Sufficiency follows from Theorem \ref{prop:extrays1}.
\end{proof}
\subsection{Proof of Theorem \ref{prop:extrays3}}

\begin{prop}\label{prop:6}
Consider $\M\in \Sminors{n}{n-1}$ such that $\rk(\M) > 1$. If $\exists~i \in [n]$ such that $\rk(\M_{ii}) < n-2$ then $\M$ does not span an extreme ray of \(\mathcal{S}_+^{n,n-1}\).\ignore{If one of the $(n-1)\times (n-1)$ principal minors has rank less than $n-2$ then $\M$ does not span an extreme ray of \(\mathcal{S}_+^{n,n-1}\).}
\end{prop}
 \begin{proof}
     Without loss of generality assume $\mathsf{rank}(\M_{11})<n-2$.
 The principal sub-matrix $\M_{11}$ and the sub-matrix $\M_{1j}$ have $n-2$ columns in common for $j=2,3,\ldots,n$. 
 These $n-2$ columns will have a rank of at most $n-3$ and adding a column to these $n-2$ common columns will make the rank of $\M_{1j}$ at most $n-2$.
Since $\det(\M_{1j}) = 0\, \forall j = 1,\ldots,n$, $\det(\M) = 0$. As all the principal minors are non-negative, this implies that $\M \succeq \bz$.
As $ \M\succeq \bz $ and $\mathsf{rank}(\M) >1 $, we can write $ \M $ as a conic combinations of rank-1 psd matrices. Hence $\M$ does not span an extreme ray.

\end{proof}

\begin{prop}\label{prop:7.1}
Let $\M \in S_{+}^{n,n-1}$ and all of its \((n-1)\times (n-1)\) principal sub-matrices have rank at least rank $n-2$. If \(\M\) has at least one full rank \((n-1)\times (n-1)\) principal sub-matrix, then $\M$ does not span an extreme ray of \(\mathcal{S}_+^{n,n-1}\).
\end{prop}
\begin{proof}
        Observe that as $\M $ has a full rank $ (n-1)\times (n-1) $ principal sub-matrix\ignore{ whose rank is $ n-1 $ this implies that $\mathsf{rank}(\M) = n-1 $ and hence} it follows that $\M$ can have at most one zero column. If $\M$ has a zero column, then $\det(\M) = 0$ and consequently $\M$ is positive semidefinite matrix.
        As $\M\succeq \bz$  and $\mathsf{rank}(\M) = n-1 $ hence it can be expressed as a conic combination of rank 1 matrices.\\
        Alternatively, if $ \M $ does not have any zero columns then we prove the existence of a rank 1 matrix, $ \varepsilon\u\u^{\top} $ where
        $ \u\in\R^n $ and $ \varepsilon\in \R_{++} $, such that $ \M +\varepsilon\u\u^{\top},\M -\varepsilon\u\u^{\top}\in \mathcal{S}_+^{n,n-1} $. 
        However, $ \M +\varepsilon \u\u^{\top}\in \mathcal{S}_+^{n,n-1} $ since $ \varepsilon\u\u^{\top}\in \mathcal{S}_+^{n,n-1} $ for all $ \varepsilon > 0 $ and $ \u\in \R^n $. 
        It suffices to show that, there exists $ \varepsilon > 0$ and $ \u \in \R^n\backslash\{\bz\}$ such that $ \M -\varepsilon\u\u^{\top}\in \mathcal{S}_+^{n,n-1} $. Alternatively, it suffices to show that 
\begin{equation}\label{prop_eq:7_1}
    \sign(\det(\M_H - \varepsilon\u_H\u_H^{\top}))\geq \sign(\det(\M_H)) \,\, \forall H\subset[n] \text{ and } |H|\leq n-1 .
\end{equation}

    Let $\mathcal{H}_{fr}\define \{J\subset[n]: \mathsf{rank}(\M_J) = |J|\text{ and } |J|\leq n-1\}$ and $\mathcal{H}_{nr} \define \{J\in[n]: \mathsf{rank}(\M_J)< |J| \text{ and } |J|\leq n-1\}$. From the matrix determinant lemma and \eqref{prop_eq:7_1}, we need to show that 
    \begin{equation}\label{prop_eq:7_4}
        \det(\M_{H} - \varepsilon\u_H\u_H^{\top}) = \det(\M_{H}) - \varepsilon\u^{\top}_H\adj(\M_H)\u_H > 0 \quad\forall H\in \mathcal{H}_{fr}, \text{ and}
    \end{equation}
    \begin{equation}\label{prop_eq:7_5}
            \det(\M_{H} - \varepsilon\u_H\u_H^{\top}) = - \varepsilon\u^{\top}_H\adj(\M_H)\u_H =0 \quad \forall H\in \mathcal{H}_{nr}.
    \end{equation}
   \ignore{\eqref{prop_eq:7_4} and \eqref{prop_eq:7_5} hold if and only if \eqref{prop_eq:7_2} and \eqref{prop_eq:7_3}.} Rewriting \eqref{prop_eq:7_4} and \eqref{prop_eq:7_5} we get,
   \begin{align}\label{prop_eq:7_2}
       \varepsilon < \min_{H \in \mathcal{H}_{fr}}\frac{\det(\M_H)}{\u^{\top}_H\adj(\M_H)\u_H}
   \end{align}
   \begin{equation}\label{prop_eq:7_3}
       \adj(\M_H)\u_H = \bz   ~\forall H \in \mathcal{H}_{nr}.
   \end{equation} 
   To see this, observe that if $\M$ has non-zero columns, and given that $\M\in \mathcal{S}^{n,n-1}_+$, then we have that the diagonal elements of $\M$ are non-zero. This implies that $\{j\}\in \mathcal{H}_{fr}\,\,\forall j \in[n]$. As $\u\neq\bz$, we have that $\exists i\in[n]$ such that $u_i \neq 0$. This implies that for $H=\{i\}$ we have that $u_i\adj(\M_H)u_i$ is non-zero as $\adj(\M_H)=1$. Therefore, the right-hand side term of the inequality \eqref{prop_eq:7_2} is positive and finite.
    Hence $\varepsilon $ has to satisfy \eqref{prop_eq:7_2} and $\u$ has to satisfy \eqref{prop_eq:7_3}. 
    Note that, if $\mathcal{H}_{nr}$ is empty, then one can find $\varepsilon>0$ for any $\u\neq \bz$ from \eqref{prop_eq:7_2}. 
    
   Let $\mathcal{J} \define \{J\in[n]: |J| = n-1 \}$ and $\mathcal{I} \define \{I\in[n]: \mathsf{rank}(\M_I) = n-2\}$. Observe that $\mathsf{rank}(\M_J) = n-1\,\, \forall J\in\mathcal{J}\backslash\mathcal{I}$. If $\mathcal{I} = \emptyset$ then $\mathcal{H}_{nr}$ is empty. To see this, observe that $\M_{J}$ is positive definite $\forall J\in\mathcal{J}$ and all of the principal sub-matrices of $\M_J$ are also positive definite if $\mathcal{I} = \emptyset$. Hence for any $\u\neq \bz$ we can find an $\varepsilon>0$ from \eqref{prop_eq:7_2} such that it satisfies \eqref{prop_eq:7_1}.
   If $\mathcal{I}\neq \emptyset$ and given that $|\mathcal{I}|\leq n-1$, then there exists a $k \in \displaystyle \bigcap_{J\in \mathcal{I}}J$ from Lemma \ref{lem:6}. Let $K \define [n]\backslash \{k\}$ and $\M_K$ is positive definite as $K\in \mathcal{J}\backslash\mathcal{I}$.\\
   Let $H\in \mathcal{H}_{nr}$ then $|H\cap K| = |H| - 1$ and $k\in H$. To see this, if $H\subset K$ then $\M_H$ is positive definite matrix and this would imply that $H \subset \mathcal{H}_{fr}$. As $\M_{H\cap K}$ is principal sub-matrix of the matrix of $\M_K$ therefore it is full rank matrix. Hence $\mathsf{rank}(\M_{H\cap K}) = |H\cap K| = |H| -1$. Observe,
   \begin{align*}
       |H| - 1 = |H\cap K| \leq \mathsf{rank}(\M_{H\cap K}) \leq \mathsf{rank}(\M_{H}) \leq |H|-1.
   \end{align*}
It follows that $\rk(\M_H)  = |H| -1 \forall H \in \mathcal{H}_{nr}$. This further implies that $\adj(\M_H)\succeq \bz$ and $\mathsf{rank}(\adj(\M_H)) =1$. Let $\m^k$ be the $k^{th}$ column in the matrix $\M$ then $\m_{H}^k$ is a column $\M_{H}$ as $k\in H$. We have,
\begin{align*}
\adj(\M_H)\M_H = \det(\M)\mathbb{I}_{n-1} = \bz, \text{ which implies }~~~~\adj(\M_H)\m_H^k = \bz.
\end{align*}
Therefore $\u = \m^k$ satisfies \eqref{prop_eq:7_3} and $\varepsilon = \dfrac{1}{2} \displaystyle\min_{H \in \mathcal{H}_{fr}}\dfrac{\det(\M_H)}{\u^{\top}_H\adj(\M_H)\u_H}$. 
 \end{proof}

\begin{prop}\label{prop:9}
    Let $\M\in \Sminors{n}{n-1}$ and $\mathsf{rank}(\M_{J}) = n-2\,\,\forall J\in [n]$ where $|J| = n-1$ and $\det(\M) < 0$. If there exists $\A,\B\in \Sminors{n}{n-1}$ such that $\M = \A + \B$, then there exists $n$ linearly independent vectors $\{\ve^i\}_{i=1}^n$ such that 
    \begin{align*}
        \M_{J_i}\ve_{J_i}^i &=\bz\\
        \A_{J_i}\ve_{J_i}^i &= \bz\\
        \B_{J_i}\ve_{J_i}^i &= \bz
    \end{align*}
    where $J_i = [n]\backslash\{i\}$ and $i\in[n]$. Furthermore, all the diagonal elements of $\M$ are non-zero. 
\end{prop}
\begin{proof}
    From Lemma \ref{lem:5}, we have that $\adj(\M_{J_i}) =\w^i{\w^i}^{\top}$ where $J_i\define[n]\backslash\{i\}$ and $i\in[n]$. 
    \begin{align*}
        \M_{J_i}\adj(\M_{J_i}) &= \det(\M_{J_i})\mathbb{I}\\
        \M_{J_i}\w^i{\w^i}^\top &= \bz\\
       (\w^i)^{\top} \M_{J_i}\w^i &= 0\\
       \M_{J_i}\w^i &= \bz\,\,\,\,(\text{As $\M_{J_i}$ positive semidefinite.)}
    \end{align*}
    Let
    \begin{align*}
        \ve^i\define\begin{cases}
            \w_j^i& \text{  for } j \in J_i\\
            0 & \text{otherwise}
        \end{cases}
    \end{align*}
Observe that $\ve_{J_i}^i = \w^i$. As $\M = \A + \B$, we have that 
\begin{align*}
    \M_{J_i} &= \A_{J_i} + \B_{J_i}\\
    (\w^i)^{\top}\M_{J_i}\w^i &= (\w^i)^{\top}\A_{J_i}\w^i + (\w^i)^{\top}\B_{J_i}\w^i\\
     0 &=(\w^i)^{\top}\A_{J_i}\w^i + (\w^i)^{\top}\B_{J_i}\w^i\\
\end{align*}
As $\A_{J_i},\B_{J_i} \in \mathcal{S}_+^{n-1}$, we have that
\begin{align*}
    (\w^i)^{\top}\A_{J_i}\w^i &=  (\w^i)^{\top}\B_{J_i}\w^i = 0\\
    \A_{J_i}\w^i &=  \B_{J_i}\w^i= \bz
\end{align*}
We now show that $\{\ve^i\}_{i=1}^n$ is linearly independent. Assume to the contrary that $\{\ve^i\}_{i=1}^n$ are linearly dependent. Let 
\begin{align*}
    \V = \begin{pmatrix}
    ({\ve}^{1})^\top\\
    ({\ve}^{2})^\top\\
    \vdots\\
    ({\ve}^{n})^\top
    \end{pmatrix}.
\end{align*}
There exists a $\u\in\R^n$ such that $\V\u =\bz$, this implies that $\mathsf{rank}(\V) = r<n$. Let the independent rows of $\V$ be $i_1,i_2,\ldots,i_r$. As $r<n$, there exists a $k \in \displaystyle\bigcap_{j=1}^r J_{i_j}$. Let $\m^{k}$ be the $k^{th}$ column of the matrix $\M$. We know that $\adj(\M_{J_{i_j}}) \m^k_{J_{i_j}} = \bz$ since index $k \in J_{i_j} \forall j \,=\,1,2,\ldots,r$ and $\M_{J_{i_j}}$ is rank-deficient matrix. This implies that $ \m^{k} \in \text{nullspace}(\V)$ and hence $(\w^k)^\top\m^{k}_{J_k} = 0\implies \adj(\M_{J_k}) \m_{J_k}^{k} = \bz$.  

But from Lemma \ref{lem:4}, we have that $ \det(\M) = - (\m^k_{J_k})^{\top}\adj(\M_{J_k}) \m_{J_k}^{k}$ and $\det(\M)<0$ which contradicts the given assumption. Hence $\V$ is a full rank matrix.

Assume to the contrary that $\M$ has at least one diagonal element which is zero. Let this element be $m_{ii} = 0$. As $m_{ii}=0$ this implies that $m_{ik} = 0$ to see this 
\begin{align*}
m_{ii}m_{kk} -m_{ik}^2 &\geq 0\,\, \because\,\, \M\in \Sminors{n}{n-1}\\
m_{ik}^2 &= 0.
\end{align*}
This means that if $\M$ has a zero column, which is the $i^{th}$ column, then $\det(\M) = 0$ which contradicts the assumption that $\det(\M) <0$. \ignore{Hence proved.} 
\end{proof}
\begin{lem}\label{lem:11}
Let $\M \in \Sminors{n}{n-1}$ and $\det(\M)<0$. Let $G_{\M}(V,E)$ be a graph associated with the matrix $\M$ with \ignore{which is defined as follows} $V = [n]$ and $ij \in E$ if $m_{ij} \neq 0$. Then $G_{\M}(V,E)$ is a connected graph.
\end{lem}
\begin{proof}
Assume to the contrary that $G_{\M}(V,E)$ is disconnected. Then $V = \bigcup_{i=1}^k V_i$  where the induced subgraph on \(V_i\) is connected and $1\leq |V_i| \leq n-1$ and \(V_i\cap V_j = \emptyset\) for \(i\neq j\). 
We can partition the matrix $\M$ into the following form,
\begin{align*}
    \M =  \begin{pmatrix}
    \M_{V_1} & \mathbf{0} & \ldots &\mathbf{0}\\
    \mathbf{0} & \M_{V_2} & \ldots & \mathbf{0} \\
    \ldots & \ldots & \ldots & \ldots \\
    \mathbf{0} & \mathbf{0} & \mathbf{0} & \M_{V_k}
    \end{pmatrix}.
\end{align*}
Each $\M_{V_k}$ is a positive semidefinite since $\M_{V_k}$ is a principal sub-matrix of size less than or equal to $n-1$. For any $\x\in \mathbb{R}^n$,\\
\begin{align*}
    \x^\top\M\x &= \Sigma_{i=1}^k \x^\top_{V_{i}}\M_{V_{i}}\x_{V_{i}}\geq 0.
    \ignore{\x^\top\M\x &\geq 0.}
\end{align*}
From this we can conclude that $\det(\M)\geq 0$. But this violates the given condition that $\det(\M) < 0$. Hence $G_{\M}(V,E)$ is a connected graph.
\end{proof}
\subsubsection{Proof of Theorem \ref{prop:extrays3}}\label{proof_thm5.3}
\begin{proof}
Let\ignore{'s assume to the contrary that $\M$ does not span an extreme ray then one can write} $\M =  \A + \B$ where $\A,\B\in\Sminors{n}{n-1}$. We will show that $\A,\B$ are positive scalar multiples of $\M$.
\ignore{Now, it follows from Corollary \ref{cor:9} and Lemma \ref{lem:10} that $(n-1)\times (n-1)$ principal sub-matrices of both $\A$ and $\B$  are of rank $n-2$ and $\det(\A),~\det(\B)<0$.}
Let $J_i = [n]\sm\{i\}$. From Proposition \ref{prop:9}, there exists linearly independent vectors $\{\ve^i\}_{i=1}^n$ such that 
\begin{equation*}
    \M_{J_i}\ve_{J_i}^i = \A_{J_i}\ve_{J_i}^i = \B_{J_i}\ve_{J_i}^i = \bz.
\end{equation*}

Now, $S \define \Bigl\{\X \in \mathcal{S}^n:\X_{J_i} \ve_{J_i}^i = \mathbf{0} \,\,\,\forall\,\,\, J_i\in \mathcal{J}\Bigl\}$. $S$ is a subspace of the vector sp space of symmetric matrices, $\mathcal{S}^n$. The number of equations that define $S$ are $n(n-1)$ and $\X$ has a total of $\dfrac{n(n-1)}{2}$ variables.
Hence $\mathsf{dim}(S)\leq \dfrac{n(n-1)}{2}$. Observe that $\M,\A,\B \in S$.\\

Let $i\in[n]$ and $\x^i$ be a column of $\X$, then we claim that,
\begin{align}\label{prop:extrays3:eq1}
    (\ve^j)^{\top}\x^i =\bz \,\,\,\,\,\forall j\in J_i.
\end{align}

To see this,
    \begin{align*}
        \mathbf{v}^{j}_{J_j} \mathbf{x}^{(i)}_{J_j}  = 0,\\
        \sum_{k\in J_j} v_k^jx_k^i = 0
\end{align*}

As $v_j^j= 0 $ from Proposition \ref{prop:9}, we have
\begin{align*}
        \sum_{k\in J_j} v_k^jx_k^i + v_j^jx_j^i = 0 = (\mathbf{v}^{j})^{\top}\x^i
    \end{align*}
and \eqref{prop:extrays3:eq1} holds.
\ignore{The last but on equality follows from the definition of $ \ve^{j} $ [Proposition \ref{prop:9}] where $ v_j^j= 0 $.}
    Let $\V$ be the matrix defined in \ref{prop:9} and the rows of the matrix are $ \{\ve^i\}_{i=1}^n $. Let $\mathbf{V}(J_i)$ be the rows of the matrix $\V$ indexed by $J_i$.
Let $\mathbf{x}^{i}$ be the $i^{th}$ row of the matrix of $\X$. It follows from \eqref{prop:extrays3:eq1}, that 
\begin{align*}
    \V(J_i)\x^i = 0.
\end{align*}
\ignore{From Proposition \ref{prop:9}, we know that $\V$ is a full rank and hence $\V(J_i)$ is a $(n-1) \times n$ full-rank matrix. But $\V(J_i)$ has $n-1$ rows and $n$ columns.} It follows from Proposition \ref{prop:9} that $\V(J_i)$ is an $(n-1) \times n$ full-rank matrix.
Select all the independent columns of $V(J_i)$ and the $ (n-1)\times(n-1)$ matrix formed by these columns is an invertible square matrix. \ignore{The leftover column let it be the $k^{th}$ column.} Denote by $k$, the column which is not part of the above invertible $(n-1)\times(n-1)$ matrix. Hence, we can express every $x_{ij}$ in terms of $x_{ik}$.

Every $x_{ij}$ of the $\X$ is scalar multiple $x_{ik}$. 
The diagonal elements of the matrix $ \X $ are non-zero, because if at least one of them is zero, let this element be the $ x_{ll} $ where $ l\in [n] $, then we would have that $ m_{ll} = 0$  this impossible from Proposition \ref{prop:9}. Therefore,
\begin{align*}
    x_{qq} &\not =  0\,\,\forall q\in[n],\\
    x_{ii} &= c_{ik-ii}x_{ik},
\end{align*}
which implies that $c_{ik-ii} \neq 0$. This further suggests that every $x_{ij}$  is scalar multiple of $x_{ii}$. Similarly, one can obtain a relation between $x_{kk}$ and $x_{ik}$. Hence,
\begin{align*}
    x_{kk} &= c_{ik-kk} x_{ik}\\
     &= \dfrac{c_{ik-kk}}{c_{ik-ii}}x_{ii}\\
     &= C_{ii-kk}x_{ii},
\end{align*}
where $ C_{ii-kk} = \dfrac{c_{ik-kk}}{c_{ik-kk}}$.

From the elimination of variables process let $P$ be the set of indices such that $x_{ij} = 0$ regardless of the matrix in $S$. Let $G_{\X}(V,E)$ \ignore{where $V = [n]$ and $ij \in E$ if $ij \not \in P$} be the graph of $\X$ as defined in Lemma \ref{lem:11}. From Lemma $\ref{lem:11}$ we know that $G_{\M}(V,E)$\ignore{, where $V = [n]$ and $ij \in E$ if $m_{ij} \neq 0$,} is a connected graph \ignore{$G_{\M}$} and is a subgraph of $G_{\X}$. \ignore{Hence} Therefore $G_{\X}$ is a connected graph.

For each edge $pq$ of the graph $G_{\X}$ one can establish a relation between $p$ and $q$. As $G_{\X}$ is connected, one can find a path between any two nodes in the graph. For example, there is a path from $i$ to $i'$ let this path be $i - p_1 - \ldots - p_l - i'$. It follows that,
\begin{align*}
    x_{p_1p_1} &= C_{ii-{p_1p_1}}x_{ii}\\
    x_{p_2p_2} &= C_{{p_1p_1}-{p_2p_2}}x_{p_1p_1}\\
    &\vdots\\
     x_{i'i'} &= C_{{p_lp_l}-{i'i'}}x_{p_lp_l}.
\end{align*}
This implies $x_{i'i'}$ is a scalar multiple of $x_{ii}$. The constants $c's$ and $ C's $ are obtained from the sub-matrix inverse of $\V$ and $\V$ is obtained from $\M$ and do not depend on $\X$. \ignore{Hence $c$'s and $ C $'s are constants regardless of $\X$.}
Hence the rank of the system $S$ is 1 and every matrix in $S$ can be written as $a\mathbf{M}'$ where $a\in \R$. Since $\A,\B,\M \in \mathcal{S}_+^{n,n-1}$, we can write $\A = a_{\A}\mathbf{M}'$, $\B = a_{\B}\mathbf{M}'$, and $\C = a_{\C}\mathbf{M}'$ where $a_{\A},a_{\B},a_{\M}\in \R$. Therefore $\A = \dfrac{a_{\A}}{a_{\M}} \M$ and $\B = \dfrac{a_\B}{a_{\M}} \M$ . Hence $\M$ spans an extreme ray of $\Sminors{n}{n-1}$. 
\end{proof}

\begin{cor}\label{cor:12}
        For $n\geq 3$,\ignore{define} \(G(n,n-1) = \dfrac{n-1}{n-2} \mathbb{I}_{n} - \dfrac{1}{n-   2}\mathbf{1}_{n}\mathbf{1}_{n}^\top\)\ignore{.Then $G(n,n-1)$} spans an extreme ray of \(\Sminors{n}{n-1}\)
\end{cor}
\begin{proof}
    It suffices to prove that \(\det(G(n,n-1)) < 0\) and all \((n-1)\times (n-1)\) principal sub-matrices of $G(n,n-1)$ are positive semidefinite and have rank \(n-2\) (Theorem \ref{prop:extrays3}).

    Observe that the eigenvalues of $G(n,n-1)$ are $\frac{-1}{n-2}$ with multiplicity 1, and $\frac{n-1}{n-2}$ with multiplicity $n-1$. We have $\det(G(n,n-1)) = -\frac{1}{n-2}\left(\frac{n-1}{n-2}\right)^{n-1} < 0.$ Furthermore, every $(n-1)\times (n-1)$ principle sub-matrix of $G(n,n-1)$ can be expressed as \(\frac{n-1}{n-2}\mathbb{I}_{n-1} - \frac{1}{n-2}\mathbf{1}_{n-1}\mathbf{1}^{\top}_{n-1} \). The eigenvalues of every  $(n-1)\times (n-1)$ principle sub-matrix of $G(n,n-1)$ are thus $0$ with multiplicity 1, and $\frac{n-1}{n-2}$ with multiplicity $(n-2)$. Consequently, every $(n-1)\times (n-1)$ principle sub-matrix of $G(n,n-1)$ is positive semidefinite with rank $(n-2)$.
\end{proof}

\Gnkextrays*
\begin{proof}
     Assume that \(G(n,k)\) does not span an extreme ray of \(\Sminors{n}{k}\). Then \(G(n,k) = \A + \B\) where \(\A,\B \in \Sminors{n}{k}\).
    Let \(J\in[n]\) and \(|J| = k+1\). It follows,
    \begin{equation*}
        (G(n,k))_J = G(k+1,k)
    \end{equation*}
    \begin{align*}
        G(n,k)_J &= \A_J + \B_J\\
        G(k+1,k) &= \A_J + \B_J\\
        j &= \min(J)\\
        \A_J &= a_{jj}G(k+1,k)\\
        \B_J &= b_{jj}G(k+1,k).
    \end{align*}
    This implies \(\A = a_{11} G(n,k)\) and \(\B = b_{11}G(n,k)\),\ignore{This is a contradiction. It follows that}and \(G(n,k)\) spans an extreme ray of \(\Sminors{n}{k}\).
\end{proof}
\NLSextremerays*
\begin{proof}

We know that $\M$ is diagonally congruent to $G(n,k)$ \cite{blekherman2022hyperbolic}. Assume to the contrary that $\M$ doesn't span an extreme ray then we have that $\M = \M_1 + \M_2$ where $\M_1,\M_2\in \calS^{n,k}_+$ and $\M_1 \neq \M_2$.
    Since $\M$ is diagonally congruent to $G(n,k)$ we have that there exists an invertible diagonal matrix such that $\M = \mathbf{D}G(n,k)\mathbf{D}$. It follows,
    \begin{align*}
        \M &= \M_1 + \M_2\\
        \mathbf{D}^{-1}\M\mathbf{D}^{-1} &= \mathbf{D}^{-1}\M_1\mathbf{D}^{-1} + \mathbf{D}^{-1}\M_2\mathbf{D}^{-1}.
    \end{align*}
    From Sylvester's law of inertia and as $\mathbf{D}^{-1}$ is a diagonal matrix we have that $\mathbf{D}^{-1}\M_1\mathbf{D}^{-1},\mathbf{D}^{-1}\M_2\mathbf{D}^{-1} \in \calS^{n,k}_+$. This contradicts corollary \ref{cor:Gnkextrays}. Hence $\M$ spans an extreme ray of $\Sminors{n}{k}$.



\end{proof}
\bibliographystyle{unsrtnat}  
\bibliography{references}

\end{document}